\documentclass[graybox]{svmult}

\usepackage{type1cm}        
\usepackage{makeidx}         
\usepackage{graphicx}        
\usepackage{multicol}        
\usepackage[bottom]{footmisc}

\usepackage{newtxtext}       %
\usepackage{newtxmath}       

\usepackage{amsfonts}
\usepackage{amsmath}
\usepackage{amssymb}
\usepackage{hyperref}
\usepackage{bm,upgreek}
\usepackage[square,numbers]{natbib}
\usepackage{theorem}
\usepackage{tensor}

\newcommand{\cV}{\mathcal{V}}

\newcommand{\bc}{\boldsymbol{c}}
\newcommand{\cT}{\mathcal{T}}
\newcommand{\ce}{\mathcal{E}}
\newcommand{\cE}{\mathcal{E}}
\newcommand{\cZ}{\mathcal{Z}}
\newcommand{\si}{\sigma}
\newcommand{\Rho} {\mathrm{P}}

\newcommand{\mbf} [1]{\mathbf{#1}}

\newcommand{\RR} {\mathbb{R}}
\newcommand{\NN} {\mathbb{N}}

\def\sideremark#1{\ifvmode\leavevmode\fi\vadjust{\vbox to0pt{\vss
 \hbox to 0pt{\hskip\hsize\hskip1em
 \vbox{\hsize3cm\tiny\raggedright\pretolerance10000
 \noindent #1\hfill}\hss}\vbox to8pt{\vfil}\vss}}}

\makeindex             

\begin{document}

\title*{Distinguished curves and first integrals on Poincar\'e-Einstein and other conformally singular geometries}
\titlerunning{Distinguished curves on conformally singular geometries}
\author{A.~Rod Gover and Daniel Snell}
\institute{A.~Rod Gover \at Department of Mathematics, The University of Auckland,
    Private Bag 92019, Auckland 1142, New Zealand \email{r.gover@auckland.ac.nz}
\and Daniel Snell \at Department of Mathematics, The University of Auckland,
    Private Bag 92019, Auckland 1142, New Zealand \email{daniel.snell@auckland.ac.nz}}

\maketitle

\abstract{We treat the problem of defining, and characterising in a
  practical way, an appropriate class of distinguished curves for
  Poincar\'e-Einstein manifolds, and other conformally singular
  geometries. These ``generalised geodesics'' agree with geodesics away
  from the conformal singularity set and are shown to satisfy natural
  ``boundary conditions'' at points where they meet or cross the
  metric singularity set.  We also characterise when they coincide
  with conformal circles. In the case of (Poincar\'e-)Einstein
  manifolds, we are able to provide a very general theory of first
  integrals for these distinguished curves.  As well as the general
  procedure outlined, a specific example is given.}

\section{Introduction}
\label{section_intro}

Geodesics and other distinguished curves play a basic and essential
role in differential geometry and its applications~\cite{Andersson-Blue,Frolov2017,Guill,HU,Wald1984}
The determination and study of these
can be enormously simplified if one has a available curve first
integrals~\cite{Andersson-Blue,Carter1968,Frolov2017,IntegrabilityKillingEqn,KillingConstantsMotion, dunajski2019conformal}.
Certain conformally singular geometries
such as Poincar\'e-Einstein manifolds have proved to have a central
place in mathematical physics, geometric scattering, the AdS/CFT
correspondence of physics, as well as in conformal geometry itself~\cite{FeffermanCharles2012Tam, GuillarmouColin2009ScoP, YangP.2008Otrv,FeffermanC2002QaPm, FeffermanC.2003Amco, GrahamC.Robin2003Smic, RobinGrahamC1999Caos}. 
It
is thus important to study the distinguished curves, and their
possible first integrals, for such geometries. However classical
theory is not directly applicable, as for these structures there is
only a well-defined metric on a dense open subset of the geometry.

We recall the notion of a conformally compact manifold. 
Let $\mathring{M}$ be the interior of a $n$-manifold with boundary
$M$. So $M= \mathring{M} \cup \partial M$ and the boundary $\partial
M$ has dimension $n-1$.  The metric $\mathring{g}$ on the interior
$\mathring{M}$ of $M$ is said to be \emph{conformally compact} if
$\mathring{g} =u^{-2} g$ where $g$ is a metric on $M$ (and so is
nondegenerate up to $\partial M$) and $u$ is a defining function for
$\partial M$, i.e. the zero locus $\mathcal{Z}(u)$ of $u$ is exactly
the boundary, so $\mathcal{Z}(u) = \partial M$, and $\mathrm{d} u \neq
0$ at all points of $\partial M$.  If in addition $\mathring{g}$ is
Einstein, then $M$ is said to be \emph{Poincar\'{e}-Einstein}. In
Riemannian signature a Poincar\'e-Einstein manifold is necessarily of
negative scalar curvature, and more generally this condition on the
curvature holds asymptotically for conformally compact manifolds as
defined here. There are variations of these statements for metrics of
other signatures. There is also a suitable notion of conformal
conformally compact for the case of zero scalar curvature. These and
the related constraints, may all be understood as special cases of the
notion of an almost (pseudo-)Riemannian manifold given in
Definition~\ref{a-r}. See~\cite{Curry-G-conformal}.

Note that for conformally compact manifolds a first problem is to
describe the right classes of distinguished curves. The usual geodesic
equation is only applicable away from the singularity of the
metric. However because the metric is singular in a conformal way, the
entire manifold $M$ has a well-defined conformal structure $[g]$. This
means that conformal circles, as defined in~\cite{Yano1938,Schouten1954,FriedSchmidt,Bailey1990a, Tod} are
well-defined and can play an important role. However in some senses
conformal circles provide a class of curves that is in a sense ``too
large''. It requires a point, velocity, {\em and acceleration}, at that point, to
determine a conformal circle.

A second problem is to treat first
integrals. Suppose that  one understands that, on a conformally compact manifold,
a certain class of unparametrised geodesics meet the boundary. Then can
these sometimes be understood, or at least simplified using first
integrals? The problem now is that geodesic first integrals are
related to the existence of so called Killing tensors~\cite{KillingConstantsMotion}, but the Killing tensor equation is not conformally invariant and is  also not well defined where the metric is conformally singular. Thus the standard theory is insufficient.

The aim of this short note is to attack and, to a considerable extent
treat, these problems. We first recall some recent advances in the
treatment of conformal circles, and the construction of conserved
quantities for these. This is done in Sections~\ref{section_dcurves_in_conf} and~\ref{section_conserved_quantities} respectively, see especially
Theorem~\ref{gst_main_c}, and
follows~\cite{GST}. That source provides a new definition of conformal
circles that is especially useful for constructing the corresponding
first integrals, and also provides a rather general theory of the
latter.  That new machinery is then used here to understand the links
between conformal circles and geodesics.  This leads us  to a way to define a class
of distinguished curves for conformally compact manifolds and their
generalisations.  The result is that for a large class of conformally
singular geometries we provide a  distinguished curve equation that
is well defined at all points and specialises to exactly the usual
unparametrised geodesic equation at points where the metric is
non-singular, see Definition \ref{keydef}.  The situation is especially beautiful and simple in the
case of Einstein and Poincar\'e-Einstein manifolds, and in those
settings we explain how to use conformal geometry to construct first
integrals for these distinguished curves.

In more detail the main results are then as follows: Theorem~\ref{main_thm} gives the ``conformal characterisation'' of metric
geodesics that leads to the definition of generalised geodesics in
definition~\ref{keydef}.  Proposition~\ref{keyprop} then shows that the
generalised geodesics extend geodesics. Proposition~\ref{app1} shows
that in the case of conformally compact manifolds the generalised
geodesics with interior points and that meet the boundary must meet it
orthogonally.  Theorem~\ref{gen_geo_einstein} gives a characterisation
of (Poincar\'e-)Einstein manifolds as an agreement between
(generalised) geodesics and conformal circles. This then leads to a
main application, namely that all the machinery from~\cite{GST} for
proliferating first integrals of conformal circles may be applied (by
specialisation) to provide first integrals of the (generalised)
geodesics of (Poincar\'e-)Einstein manifolds. This is taken up briefly in Section~\ref{section_conserved_quantities}.

Some remarks on the notation and conventions used in this paper are in order. 
We work on a (pseudo-)Riemannian or conformal manifold of dimension $n \geq 3$. 
We will always assume that our manifold is connected.
The Riemannian or conformal structure is allowed to have any signature $(p,q)$, although some aspects of the paper will only be relevant when the signature is strictly pseudo-Riemannian (namely $pq \neq 0$).
We will use Penrose's abstract index notation~\cite{Penrose-Rindler-v1} for sections of vector bundles.
In this notation, lower case Latin indices are used to indicate the type of a section of a tensor bundle in the following way.
The tangent bundle is denoted by $\cE^a$ and the cotangent bundle is denoted by $\cE_b$.
Tensor products are then denoted by an appropriate combination of indices.
So for example, $v^a$ is a vector field, $\omega_{bcd}$ is a 3-form, and $T^a {}_b$ is an endomorphism of the tangent (or cotangent) bundle. 
The notation $\cE$ is used for the trivial bundle, so that $\Gamma(\cE) = C^\infty(M)$.
Square brackets around indices denotes the completely anti-symmetric part of the given tensor, and round brackets around indices denotes the completely symmetric part.
Given a vector bundle $\cV$, $\Lambda^k \cV$ will denote the subbundle of $\otimes^k \cV$ consisting of the totally anti-symmetric tensors. 
Similarly, $S^\ell \cV$ denotes the subbundle of totally symmetric tensors. 
Finally we will occasionally employ a somewhat informal ``wedge'' notation for convenience.
If $S^{a_1\cdots a_k} \in \Gamma(\Lambda^k \cV)$ and $T^{b_1\cdots b_\ell} \in \Gamma(\Lambda^\ell \cV)$, we define
\begin{equation} \label{wedge_notation}
  S \wedge T := S^{[a_1\cdots a_k} T^{b_1\cdots b_\ell]}.
\end{equation}
All objects and functions are assumed smooth ($C^\infty$).

\subsection{Acknowledgements} ARG gratefully acknowledges support from the Royal
  Society of New Zealand via Marsden Grant 16-UOA-051.

\section{Conformal geometry and conformal tractor calculus} \label{section_tractor_calc}

We review the basic theory of conformal tractor calculus that will be
required for our purposes.

Conformal density bundles are a family of line bundles which arise naturally in conformal geometry. 
Recall that any manifold $M$ is equipped with the oriented line bundle $\mathcal{K} := (\Lambda^n TM)^2$.
Conformal density bundles are then various roots of this bundle.
For $w\in\RR$, let 
\begin{equation} \label{conf_density_defn}
  \cE[w] := \mathcal{K}^{\frac{w}{2n}}.
\end{equation}
Given a vector bundle $\cV$, we define $\cV[w] := \cV \otimes \cE[w]$.

The conformal structure $(M, \bm{c})$ determines a distinguished section of $\cE_{(ab)}[2]$ which we denote $\mbf{g}_{ab}$ and call the \emph{conformal metric}.
This section is characterised by the fact that any choice of metric $g$ in the conformal class may be realised as $g = \sigma_g ^{-2} \mbf{g}$ for some $\sigma_g \in \Gamma (\cE_+ [1])$ depending on $g$.
Conversely, given a non-vanishing 1-density $\sigma \in \Gamma(\cE_+ [1])$, $g := \sigma^{-2} \mbf{g}$ is a metric in the conformal class. 
In light of this bijective correspondence between metrics in the conformal class and sections of the bundle $\cE_+ [1]$, we call sections of this bundle \emph{scales}, and will often refer to a choice of metric to facilitate computations as a \emph{choice of scale}.
We shall use $\mbf{g}$ to raise and lower indices on a conformal manifold.
This is very similar to the way that the Riemannian metric is used, since the conformal metric is preserved by \emph{any} Levi-Civita connection from the conformal class.
There is however the small caveat that raising and lowering indices via the conformal metric and its inverse does change the weight of the section on which it acts:
\begin{equation}
  \mbf{g}_{ab} : \cE^a \to \cE_b [2] \hspace{1em} \textrm{by} \hspace{1em} v^a \mapsto \mbf{g}_{ab} v^a.
\end{equation}

There are many ways to realise the standard conformal tractor bundle; here we shall present one that has the advantage of introducing an equation we will use again later.

It is a well-known fact that the following equation is conformally invariant
\begin{equation}\label{ae_eqn}
  \nabla_{(a} \nabla_{b)_{0}} \sigma + P_{(ab)_{0}} \sigma = 0,
\end{equation}
where $\sigma \in \cE[1]$ is a conformal density.  This equation will
be called the \emph{Almost-Einstein equation} (A.E.), since a solution
$\sigma$ defines an Einstein metric everywhere $\sigma \neq 0$.  The
above equation may be \emph{prolonged} to give an equivalent first
order closed system; see~\cite{Curry-G-conformal} for a more detailed
description of this process.  The prolongation moreover yields a
linear connection on a certain vector bundle.  The solutions of the
A.E. equation are in bijective correspondence with sections of this
vector bundle which are parallel for this connection.  This vector
bundle is the \emph{tractor bundle}, denoted $\cT$ and the connection
is the \emph{tractor connection}, denoted simply $\nabla$ or sometimes
$\nabla^{\cT}$ to avoid ambiguity.

A choice of metric $g\in \bm{c}$ in the conformal class induces an isomorphism 
\begin{equation} \label{cT_isom}
  \cT \overset{g}{\cong} \cE[1] \oplus \cE_a[1] \oplus \cE[-1].
\end{equation}

In terms of this, the connection acts as
\begin{equation} \label{cT_connection}
  \nabla^{\cT} _a 
  \begin{pmatrix}
    \sigma \\ \mu_b \\ \rho
  \end{pmatrix}
  \overset{g}{=}
  \begin{pmatrix}
    \nabla_a \sigma - \mu_a\\
    \nabla_a \mu_b + \mbf{g}_{ab}\rho + P_{ab}\sigma \\
    \nabla_a \rho - P_{ab} \mu^b
  \end{pmatrix},
\end{equation}
where we are again using a choice of metric to explicitly write down the formula.

The tractor bundle also comes equipped with a symmetric non-degenerate bilinear form, defined by
\begin{equation} \label{cT_metric}
  h(V,V') \overset{g}{=}
  \begin{pmatrix}
    \sigma & \mu_a & \rho
  \end{pmatrix}
  \begin{pmatrix}
    0 & 0 & 1\\
    0 & \mbf{g}^{ab} & 0\\
    1 & 0 & 1\\
  \end{pmatrix}
  \begin{pmatrix}
    \sigma' \\ \mu'_b \\ \rho'
  \end{pmatrix}.
\end{equation}

In abstract indices, tractors will be denoted with upper case Latin indices. 
Hence the tractor metric will be written $h_{AB}$ and is a section of $\cE_{(AB)}$. 
The tractor metric identifies $\cE^A$, the tractor bundle, with its dual $\cE_A$. 

There is a conformally-invariant bundle map $\cE[-1] \hookrightarrow \cT$.
We denote this map by $X$ and refer to it as the \emph{canonical tractor} or \emph{position tractor}.
It plays a central role in our characterisation of conformal circles.
The map $X$ may be thought of as a section of $\cE_A [1]$. 
Raising the tractor index via the tractor metric, we obtain a section $X^A \in \Gamma(\cE^A [1])$, which coincides with the conformally invariant projection $X^A : \cE_A \to \cE[1]$. 

This idea can be extended to give a convenient notation for working
with tractors.  We have already seen that there is a map which inserts
the subbundle $\cE[-1]$ into the tractor bundle.  Given a choice of
scale there are corresponding sections mapping the other two slots into the triple.
We emphasise these are not conformally invariant.
For a section
$V^A \in \Gamma(\cE^A)$, we write $V^A = \sigma Y^A + \mu^a Z^A _a +
\rho X^A$ following~\cite{G-PetersonCMP}.
Here, $Y^A \in \Gamma(\cE^A [-1])$ and $Z^A _a \in \Gamma(\cE^A _a [1])$.

The final element of the standard theory of tractor calculus we shall require is the second order differential operator $D : \cE[1] \to \cT$ such that $\nabla_a ^\cT D_A \sigma = 0$ if, and only if, $\sigma$ solves the A.E. equation~\eqref{ae_eqn}.
It is a way of encoding a choice of scale in a tractor. 
In a choice of scale and using the tractor projector notation of the introduced in the previous paragraph, one has  
\begin{equation} \label{thomas_D_defn}
  D_A \sigma \overset{g}{=} 
  \sigma Y_A +  \nabla_b \sigma  Z^b _A - \frac{1}{n}(\Delta + J) \sigma X_A,
\end{equation}
where $J$ is the conformal metric trace of the Schouten tensor of $g$.
We say that a tractor $I^A \in \Gamma(\cE^A)$ is a \emph{scale tractor} if $I^A = D^A \sigma$ for some $\sigma \in \Gamma(\cE_+[1])$ and $I^A$  is nowhere zero.

The canonical tractor $X^A$ mentioned previously recovers the section $\sigma$ from such a scale tractor. 
Explicitly, $X^A D_A \sigma = \sigma$.

\section{Distinguished curves in conformal geometry} \label{section_dcurves_in_conf}

For us, a curve will be a smooth, regular map $\gamma : I \to M$, where $I \subset \RR$ is an interval.
The distinguished curves of a conformal manifold are divided into two classes: the null geodesics (if the manifold is strictly pseudo-Riemannian) and the (necessarily non-null) conformal circles.

We recall here the main results that we will require from~\cite{GST}.
Rather than working directly with the velocity and acceleration of the curve, it turns out to be more natural to work with weighted versions of these. 
These are constructed as follows.
The velocity $u^a$ of the curve $\gamma$ determines a scale $\sigma_u \in \Gamma(\cE_+[1]|_{\gamma})$ along the curve via 
\begin{equation}
  \sigma_u :=
  \begin{cases}
    \sqrt{\mbf{g}_{ab} u^a u^b} & \textrm{if } u^a \textrm{ is spacelike,}\\
    \sqrt{-\mbf{g}_{ab} u^a u^b} & \textrm{if } u^a \textrm{ is timelike.}\\
  \end{cases}
\end{equation}
Using this, we define the weighted velocity, $\mbf{u}^a \in \Gamma(T\gamma[-1])$ by $\mbf{u}^a := \sigma_u ^{-1} u^a$.
This weighted velocity does not depend on the parametrisation of the curve.
Moreover, we define the weighted version of the acceleration, $\mbf{a}^b \in \Gamma(T \gamma [-2])$ by $\mbf{a}^b := \mbf{u}^a \nabla_a \mbf{u}^b $.
Note that the while weighted velocity is conformally invariant, the weighted acceleration is not.

The original conformal circle equation is a second-order ODE in the velocity of the curve; see e.g.~\cite{Bailey1990a}.
With the above definitions, we can give an equivalent presentation of the conformal circle equation, which is essentially a weighted version of equation (7) from~\cite{Bailey1990a}.

\begin{lemma}
  Let $\gamma$ be an oriented nowhere-null curve on $(M,\bc)$. Then $\gamma$ is a conformal circle  if and only if its weighted  velocity $\mbf{u}^a$   and acceleration $\mbf{a}^a$ satisfy the conformally invariant equation
  \begin{align}\label{eq-proj-inv-wt}
    \left( \mbf{u}^c \nabla_c \mbf{a}^{[a} \right) \mbf{u}^{b]} & = \pm \mbf{u}^c \Rho_c{}^{[a} \mbf{u}^{b]} \, , & \mbox{whenever $\mbf{u}^a\mbf{u}_a=\pm 1$,}
  \end{align}
  or equivalently,
  \begin{align}\label{eq-proj-inv-wt2}
    \mbf{u}^b \nabla_b \mbf{a}^a & = \pm \mbf{u}^b \Rho_b {}^a  - ( \Rho_{bc} \mbf{u}^b \mbf{u}^c \pm \mbf{a} \cdot \mbf{a} ) \mbf{u}^a \, , & \mbox{whenever $\mbf{u}^a\mbf{u}_a=\pm 1$,}
  \end{align}
  for any $g \in \bc$ with Levi-Civita connection $\nabla$.
\end{lemma}

\begin{proof}
  See~\cite{GST}.
\end{proof}

The previous lemma may be reformulated in terms of tractors, using a moving incidence relation; this is a main result from~\cite{GST} for conformal circles. 

\begin{theorem}\label{gst_main_c}
  On a pseudo-Riemannian or conformal manifold a nowhere null curve
  $\gamma$ is an oriented conformal circle if and  only if along  $\gamma$ there is a parallel $3$-tractor $0\neq \Sigma \in \Gamma(\Lambda^3 \cT|_\gamma)$ such that
  \begin{align}\label{main_c_eqn}
    X\wedge \Sigma=0.
  \end{align}
  For a given oriented conformal circle $\gamma$ the $3$-tractor $\Sigma_\gamma$
  satisfying~\eqref{main_c_eqn} is unique up to multiplication by a
  positive constant, and unique if we specify $|\Sigma_\gamma|^2=-1$ when $\gamma$ is spacelike, or $|\Sigma_\gamma|^2=1$ when $\gamma$ is timelike.
\end{theorem}
\begin{proof}
We sketch here the proof, but for details see \cite{GST}.
First, following \cite{BEG}, for a non-null curve $\gamma$, we define the \emph{velocity tractor} and \emph{acceleration tractor} respectively by
\begin{equation}
  U^A := \sigma_u \mbf{u}^a \nabla_a (\sigma_u^{-1} X^A), \hspace{1em} \textrm{and} \hspace{1em} A^B := \sigma_u \mbf{u}^a \nabla_a U^B.
\end{equation}
These are conformally invariant by construction. But they are not parametrisation independent.
Then form
\begin{equation} \label{Sigma}
  \Sigma^{ABC} :=6 \sigma_u^{-1} X^{[A} U^B A^{C]} ,
  = \pm 6 \mbf{u}^c X^{[A} Y^B Z^{C]} {}_c + 6 \mbf{u}^b \mbf{a}^c X^{[A} Z^B {}_b Z^{C]} {}_c .
\end{equation}
This is parametrisation independent and conformally invariant.
Moreover it is parallel along $\gamma$ if, and only if,~\eqref{eq-proj-inv-wt2} holds.

Now if $\Sigma \in \Gamma(\Lambda^3 \cT)$ takes the form~\eqref{Sigma}, then $X \wedge \Sigma = 0$. 
On the other hand, if $\Sigma \in \Gamma(\Lambda^3 \cT)$ is parallel along the curve and satisfies $X \wedge \Sigma = 0$, then $\Sigma$ must take the form~\eqref{Sigma}. This completes the proof.
\end{proof}

We wish to remark here that the velocity and acceleration tractors, and therefore the 3-tractor $\Sigma$ as defined above, can be defined for \emph{any} non-null curve $\gamma$;
one does not require that the curve is a conformal circle. Thus $\Sigma$ is a fundamental conformal invariant of such unparametrised curves. This
observation will be relevant when we consider the distinguished curves
of Riemannian geometry in the following section.  

We shall also require the derivative of such a $\Sigma$ along the curve, and so we record that formula here for later reference:
\begin{equation}\label{deriv_Sigma}
    \mbf{u}^d \nabla_d \Sigma^{{A} {B} {C}}  =  6 \left( \mbf{u}^d \nabla_d \mbf{a}^c \mp \mbf{u}^d \Rho_d{}^c \right) \mbf{u}^b \, X^{[{A}}  Z^{{B}}{}_b Z^{{C}]} {}_c\, ,  \mbox{whenever $\mbf{u}^a \mbf{u}_a = \pm 1$,}
\end{equation}

As a consequence, we obtain a parametrisation-independent version of Proposition~3.3 from~\cite{Bailey1990a}.

\begin{proposition}\label{bailey_eastwood}
  Let $(M, \bm{c})$ be a conformal manifold and $\gamma$ a non-null, oriented curve in $M$.
Then $\gamma$ is an unparametrised conformal circle if, and only if, $\gamma$ is an unparametrised geodesic for some metric in the conformal class and $\mbf{u}^a P_a {}^b \propto \mbf{u}^b$, where $P$ is the Schouten tensor for the given metric.
\end{proposition}

\begin{proof}
  First, suppose that $\gamma$ is a conformal circle. 
  Locally, there exists a metric in the conformal class for which $\gamma$ is an affine geodesic~\cite{Bailey1990a}. 
  (A proof of this fact was sent to us in a private communication from Michael Eastwood. 
  It will appear a forthcoming article of his with Lenka Zalabov\'{a} \cite{MikeLenka}.)
  Let $\sigma$ be the scale corresponding to this metric.
  Now, working in the scale $\sigma$,~\eqref{deriv_Sigma} becomes
  \begin{equation}\label{reduced_deriv_Sigma}
    \mbf{u}^d \nabla_d \Sigma^{{A} {B} {C}}  =   \mp 6 \mbf{u}^d \Rho_d{}^c \mbf{u}^b \, X^{[{A}}  Z^{{B}}{}_b Z^{{C}]} {}_c\, ,
  \end{equation}
where $\Rho_d{}^c$ is the Schouten for this metric.
  But since $\gamma$ is a conformal circle, the left side of the previous display must be zero by Theorem~\ref{gst_main_c}.
  Hence $\mbf{u}^d P_d {}^c \propto \mbf{u}^c$.
  Conversely, if there is a scale in the conformal class for which $\gamma$ is a geodesic, then working in the assumed scale, $\mbf{a} \propto \mbf{u}$ and hence~\eqref{deriv_Sigma} will again take the form~\eqref{reduced_deriv_Sigma}.
  Hence if $\mbf{u}^d P_d {}^c \propto \mbf{u}^c$ then $\mbf{u}^d \nabla_d \Sigma^{ABC} = 0$, and so $\gamma$ is a conformal circle by Theorem~\ref{gst_main_c}.
\end{proof}

\section{Metric geodesics via conformal tractor machinery} \label{section_riem_geos}

Here we give first a characterisation of geodesics that, in its nature,
is ``as conformal as possible'' and which also clarifies the relation between geodesics and conformal circles.
We have already noted that the section $\Sigma \in \Gamma(\Lambda^3 \cT|_\gamma)$ from display~\eqref{Sigma}, and defined by $\Sigma^{ABC} :=6 \sigma_u ^{-1} X^{[A} U^B A^{C]}$, is determined by  \emph{any} non-null unparametrised curve $\gamma$ in a conformal manifold.
While initially defined and used in our work characterising conformal circles, the following theorem shows that this 3-tractor also turns out to provide a way to characterise Riemannian distinguished curves.

\begin{theorem}\label{main_thm}
  Let $(M,g)$ be a pseudo-Riemannian manifold. 
  A non-null curve $\gamma$ is an unparametrised geodesic of $g$ if, and only if, the scale tractor $I=D\si$ (where $g=\si^{-2}\mbf{g}$) satisfies  $I \wedge \Sigma = 0$.
\end{theorem}

\begin{proof}
  Suppose that $\gamma$ is an unparametrised geodesic for $g$.
  Then we may form $\Sigma^{ABC}$.
  Let $\sigma$ be the scale determined by $g$, i.e. $g = \sigma^{-2} \mbf{g}$, and let $I := \frac{1}{n} D \sigma$.
  Then, working in the scale $\sigma$, we have that $\nabla \sigma = 0$ (see e.g.~\cite{Curry-G-conformal}) and so $I^A = \sigma Y^A + \rho X^A$, for some $\rho \in \cE[-1]$
    whose explicit form will not be needed.
  Then 
  \begin{align*}
    I \wedge \Sigma &= \left(\sigma Y^A + \rho X^A \right ) \wedge \left(6 \mbf{u}^a X^{[B} Y^{C} Z^{D]} _a \mp 6 \mbf{u}^a \mbf{a}^b X^{[B} Z^C _a Z^{D]} _b \right)\\
                    &= \mp 6 \sigma \mbf{u}^a \mbf{a}^b Y^{[A} X^B Z^C _a Z^{D]} _b.
  \end{align*}
  Note that in the scale $\sigma$, $\mbf{a}^c = \sigma^{-2} a^c$ and $\mbf{u}^c = \sigma^{-1} u^b$, and hence $\mbf{u} \propto \mbf{a}$ if, and only if $a \propto u$.
  But the latter is the case, since $\gamma$ is an unparametrised geodesic, so $I \wedge \Sigma = 0$.

Conversely, suppose that the  scale tractor $I$ satisfies  $I \wedge \Sigma = 0$.
     Then in the scale $\sigma$, $I \wedge \Sigma$ again takes the form of the above display.
  Since the tractor projectors are linearly independent pointwise, $I \wedge \Sigma = 0$ implies that $\mbf{a} \propto \mbf{u}$, which in turn implies that $a^b \propto u^b$ in the scale $\sigma$.
  Hence $\gamma$ is an unparametrised geodesic for the metric $g$.
\end{proof}

Theorem~\ref{main_thm} suggests a way to generalise the geodesic
equation to a larger class of structures, which includes conformally
compact manifolds.

Recall that on an Einstein manifold with Einstein scale $\sigma \in
\Gamma(\cE_+[1])$, the corresponding scale tractor $I_A := \frac{1}{n}
D_A \sigma$ is parallel.  More generally, such a scale tractor may be
used to define a slight weakening of the notion of a pseudo-Riemannian
manifold.  Note that if a scale tractor $I_A$ is nowhere zero, then
$\sigma=X^AI_A$ is non-zero on an open dense set. Therefore that same
open dense set possesses a pseudo-Riemannian metric defined by
$\mathring{g} := \sigma^{-2} \mbf{g}$.

\begin{definition} \label{a-r}
  We shall call $(M ,\bm{c}, \sigma)$ an \emph{almost pseudo-Riemannian manifold} if $I := \frac{1}{n} D\sigma$ is nowhere-zero.
\end{definition}

This definition encompasses conformally compact structures, the open
dense set being the interior, uncompactified manifold, and
$\sigma^{-2} \mbf{g}$ recovers the original metric. But similar
comments apply to suitable conformal type compactifications of metrics
in any signature and also the case where the scalar curvature is zero~\cite{Curry-G-conformal}. 
And it is also interesting to consider closed manifolds with such a almost pseudo-Riemannian structure~\cite{GRiemSig, Curry-G-conformal}.
For any such geometry there is a natural extension of the (unparametrised) geodesic equation to any points where the metric is singular, as follows.

\begin{definition} \label{keydef}
  Let $(M,\bm{c}, \sigma)$ be an almost pseudo-Riemannian manifold. 
  We will say that an unparametrised curve $\gamma$ in $M$ is a \emph{generalised geodesic} if
  \begin{equation} \label{gen_geodesic_equation}
    I \wedge \Sigma = 0,
  \end{equation}
  where $I := \frac{1}{n} D \sigma$ and $\Sigma$ is the 3-tractor defined in Section \ref{section_riem_geos}.
\end{definition}

Observe that the equation  (\ref{gen_geodesic_equation}) is
well defined globally on an almost pseudo-Riemannian manifold, and in
particular through the zero locus $\cZ(\sigma)$ (of the generalised scale $\si$)
where the metric is singular. However in any neighbourhood where $\si$ is non-vanishing
the generalised geodesic equation
recovers the usual geodesic equation. Thus generalised  geodesics extend geodesics to(/through) the metric singularity set.

\begin{proposition}\label{keyprop}
  Let $(M, \bm{c}, \sigma)$ be an almost Riemannian manifold. 
  Suppose that $\gamma$ is a smooth unparametrised curve  on $M$
  that is a geodesic for $g=\si^{-2}\bm{g}$ on
  $M\setminus \mathcal{Z}(\si)$.
  Then $\gamma$ is a generalised geodesic of $(M,\bm{c},\si)$.
\end{proposition}

\begin{proof}
  Since   $\gamma$ satisfies the
  geodesic equation on $M\setminus \mathcal{Z}(\si)$.  This means that $I \wedge
  \Sigma = 0$ on that open and dense set, as seen in
  Theorem~\ref{main_thm}.  But by smoothness, this is then also true
  on the closure $M$, and hence $I \wedge \Sigma = 0$ everywhere, and
  $\gamma$ is a generalised geodesic.
\end{proof}

Specialising to conformally compact manifolds we have the following
result.  It turns out that the generalised geodesics which extend
smoothly to the boundary cannot do so arbitrarily.  In fact, they may
only meet the boundary in a very controlled way (that generalises the
well known special case of compactified hyperbolic space).

\begin{proposition}\label{app1}
  On a conformally compact manifold, any generalised geodesic which
  extends to the boundary meets the boundary orthogonally.
\end{proposition}

\begin{proof}
  Since $\partial M = \cZ(\sigma)$, at the boundary, 
  \begin{equation}\label{I_along_bdry}
    I^A |_{\partial M} = \nabla^a \sigma Z^A {}_a -\frac{1}{n} \Delta \sigma X^A,
  \end{equation}
and $\nabla \si$ is non-vanishing at all boundary points.
  Thus
  \begin{align*}
    (I^A \wedge \Sigma^{BCD})|_{\partial M} &= \left(\nabla^a \sigma Z_a {}^A -\frac{1}{n} \Delta \sigma X^A\right ) \wedge \left(6 \mbf{u}^d X^{[B} Y^{C} Z^{D]} _d \mp 6 \mbf{u}^c \mbf{a}^d X^{[B} Z^C _c Z^{D]} _d \right)\\
                                    &= 6 \mbf{u}^d \nabla^a \sigma Z^{[A} {}_a X^B Y^C Z^{D]} {}_d .
  \end{align*}
  
  Thus $I\wedge \Sigma = 0$ implies that $\mbf{u} \propto \nabla \sigma$.
  But $\nabla \sigma$ is a conormal to the boundary, and therefore $\mbf{u}$ is orthogonal to the boundary as well.
\end{proof}
\section{(Poincar\'{e}-)Einstein structures}\label{section_p_e_structures}

We have already noted that a solution to~\eqref{ae_eqn} determines an Einstein metric
almost-everywhere.  Hence following e.g.~\cite{gover2004almost}, we
make the following definition.

\begin{definition}\label{ae-structure}
  We say that $(M,\bm{c},\sigma)$ is an \emph{Almost-Einstein structure} if $\sigma$ is a nontrivial solution to~\eqref{ae_eqn}.
\end{definition}

In this case the scale tractor $I$ is parallel, and thus nowhere zero. Thus the structure is almost pseudo-Riemannian and $\si$
is non-vanishing on an open, dense set. On the set where $\sigma$ is non-zero the metric defined by $ g:= \sigma^{-2} \mbf{g}$ is Einstein.
Proposition~\ref{bailey_eastwood} now gives the following result.

\begin{corollary}\label{geo_is_conf_circ}
  Suppose $(M,\bm{c},\sigma)$ is an almost-Einstein manifold.  If
  $\gamma$ is an generalised geodesic then $\gamma$ is an
  unparametrised conformal circle.
\end{corollary}

\begin{proof}
  Note that if $\gamma$ is a geodesic and in addition $\mbf{u}^b P_b
  {}^c \propto \mbf{u}^c$, then $\gamma$ is a conformal circle by
  Proposition~\ref{bailey_eastwood}.  But this immediately implies the result on
  $M\setminus \mathcal{Z}(\si)$, since if $M$ is Einstein, $P_{bc}
  \propto g_{bc}$. Then the full result follows by continuity.
\end{proof}

The above also admits a converse, yielding the following
characterisation of Einstein manifolds and Poincar\'e-Einstein
manifolds.

\begin{theorem}\label{gen_geo_einstein}
  Let $(M, \bm{c}, \sigma)$ be a (conformally compact or)
  pseudo-Riemannian manifold.  Then every
  (generalised) geodesic of $M$ is a conformal circle if, and only if,
  $(M, \bm{c}, \sigma)$ is (Poincar\'{e}-)Einstein.
\end{theorem}

\begin{proof}
  One direction is exactly Corollary~\ref{geo_is_conf_circ}.
  For the converse, we must show that if every geodesic of the manifold is a conformal circle, then the metric $g=\si^{-2}\bm{g}$ on $M\setminus \mathcal{Z}(\si)$ is Einstein.
  Note that for any $p \in M\setminus \mathcal{Z}(\si)$ and any $\mbf{u}^a \in \cE^a[-1]$, there is an unparametrised geodesic passing through $p$ with velocity $\mbf{u}^a$. 
  The statement we wish to prove now essentially reduces to a question of linear algebra. 
  Fix an orthonormal basis $\{e_1, \ldots e_n\}$ for $T_p M[-1]$, by which we mean $g(e_i, e_i) = \pm 1$ for all $i$ according to the signature of the metric.
  Recalling corollary~\ref{bailey_eastwood}, each geodesic  is an  unparametrised
  conformal circle only if  $\mbf{u}^a P_a {}^b
  \propto \mbf{u}^b$. Where $P_a {}^b$ is the Schouten tensor of the metric $g$.
  Working in the above basis, we see immediately
  that, $P(e_i, e_j) = 0$ when $i\neq j$, and hence when written as a
  matrix, $P = \operatorname{diag} \{ \lambda_1, \lambda_2, \ldots,
  \lambda_n \}$.
  But $\mbf{u}^a P_a
         {}^b \propto \mbf{u}^b$ for \emph{all} $\mbf{u} \in T_p
         M[-1]$.  By taking $\mbf{u} = e_i + \alpha e_j$, for $\alpha
         \in \RR$ as $i,j$ range over $1,2,\ldots, n$, we conclude
         that $\lambda_i = \lambda_j$ for all $i,j$.  Hence $P =
         \lambda \operatorname{Id}$ (where $P$ is the matrix of $P_i{}^j$ in this basis) and the proposition follows.
\end{proof}

This theorem suggests one way to approach the study and construction
of conserved quantities on Einstein manifolds.  Namely, one may view
the (generalised-)geodesics of an (Poincar\'e-)Einstein manifold as conformal circles that in addition satisfy $I\wedge \Sigma=0$. Then, 
using the theory from~\cite{GST}, conformal BGG equations become
available as a source of ``symmetries'' to construct first integrals.
We will touch on this idea in the following section.

\section{Conserved quantities} \label{section_conserved_quantities}

One of the main applications of a tractor characterisation of distinguished curves is the development of a general theory of first integrals for such curves.
We review a general method for producing conserved quantities for conformal circles, and refer the reader to~\cite{GST} for a more detailed treatment. 
Suppose $\gamma$ is a distinguished curve.
Let $S \in \Gamma(\cV)$ where $\cV$ is some (suitably) irreducible part of $S^k \Sigma$ or for $k\in \NN$, where by $S^k \Sigma$ we mean $\Sigma \odot \Sigma \odot \cdots \odot \Sigma$, with $\Sigma$ repeated $k$ times.
Then since $S$ will only involve $\Sigma$ and possibly the tractor metric, one has 
\[
  \mbf{u}^a \nabla_a S = 0.
\]
Suppose moreover that $T$ is a section of the dual bundle $\cV^*$.
Then, if $\mbf{u}^a \nabla_a T = 0$, we will have
\[
  \mbf{u}^a \nabla_a \langle S, T \rangle = 0,
\]
where $\langle \cdot, \cdot \rangle$ denotes the dual pairing between
sections of $\cV$ and $\cV^*$.  Asking for such sections $T$ that are
parallel is restrictive, but far from fatally.  Solutions to so-called
BGG equations are in bijective correspondence with sections of tractor
bundles which are parallel for a modified tractor connection - the
\emph{prolongation connection} of~\cite{HSSS}.  In the case that a
solution is a \emph{normal} solution~\cite{CapA2012NBSA,CGH-Duke}, this
modified connection agrees with the standard tractor connection.  Thus
if $T$ arises as the section of some tractor bundle corresponding to a
normal solution to a BGG equation, we will have $\mbf{u}^a \nabla_a T
= 0$ as desired. For some BGG equations all solutions are normal.  On
conformally flat manifolds all solutions are normal. Most known
superintegrable geometries are conformally flat. Then in fact the
examples in \cite{GST} show that surprisingly the first integrals
found do not actually require normality. We treat such an example
here.

\subsection{An example in the Poincar\'{e}-Einstein case}

We have already seen that on a (Poincar\'{e}-)Einstein manifold any generalised geodesic is a conformal circle, that satisfies in addition $I\wedge \Sigma=0$.

As a consequence of Theorem~\ref{gen_geo_einstein}, we see that
conformal geometry machinery may be used to proliferate first
integrals of the generalised geodesics on (Poincar\'{e}-)Einstein
manifolds. (Or more more generally almost-Einstein manifolds. For
simplicity here we simply discuss the Poincar\'e-Einstein case.) 

\begin{proposition}
  Let $(M, \bm{c}, \sigma)$ be a (Poincar\'{e}-)Einstein manifold.
  The first integrals of generalised geodesics of $M$ are also first integrals of conformal geodesics, and may therefore be produced using solutions to BGG equations, c.f. Theorem 6.3 of~\cite{GST}.
\end{proposition}

Recall the conformal Killing form equation:
\begin{equation}
  \operatorname{tf}(\nabla F) \in \Gamma(\Lambda^{k+1} T^*M),
\end{equation}
where $\operatorname{tf}$ means the metric trace-free part of the
given tensor.  Normal solutions to this equation are in bijective
correspondence with sections of $\Lambda^3
\cT^*$ that are parallel for the usual tractor connection \cite{G-Sil-ckforms}.  In this particular case, it turns out
that the normality assumption is not even required.

\begin{theorem}\label{p_e_example}
  Let $(M, \bm{c}, \sigma)$ be a Poincar\'{e}-Einstein manifold.
  Suppose that $\gamma$ is a generalised geodesic, and form $\Sigma^{ABC}$ according to~\eqref{Sigma}.
  Finally, suppose that \( k_{ab}\in \Gamma(\ce_{[bc]}[3]) \) is a conformal Killing-Yano
  2-form of the conformal manifold $(M,\bm{c})$, i.e.\ $k_{ab}$ satisfies
  \begin{align}\label{eq-CKY2}
    \nabla_{a}k_{bc} & =\nabla_{[a}k_{bc]} -\frac{2}{n-1}\mbf{g}_{a[b}\nabla^pk_{c]p} \, .
  \end{align}
  Write \(\Gamma( \Lambda^3\cT^* )\ni \mathbb{K}_{ABC}
  :=L(k) \) where the BGG splitting operator  $L:
    \ce_{[bc]}[3]\to \mathcal{E}_{[ABC]}$ maps a conformal Killing 2-form invariantly to its corresponding 3-tractor \cite{CGH-Duke,HSSS}.
  Then expression
  \[
    \Sigma^{ABC} \mathbb{K}_{ABC}
  \]
  equivalently,
  \[
    \mathbf{u}^a \mathbf{a}^b k_{ab} \mp \frac{1}{n-1} \mathbf{u}^a \nabla^p k_{pa},
  \]
  is a first integral of the generalised geodesic $\gamma$, where   $\mbf{u}^a$ denotes its  weighted velocity, with $\mbf{u}^a \mbf{u}_a = \pm 1$, depending on whether the curve $\gamma$ is spacelike or timelike respectively.
\end{theorem}

\begin{proof}
  According to Corollary~\ref{geo_is_conf_circ}, any generalised
  geodesic of the (Poincar\'e-)Einstein manifold is in fact a
  conformal circle for the underlying conformal manifold.  Thus the
  curve $\gamma$ is also a conformal circle.
   Having established this, the result follows at once from
  Theorem 6.8 of~\cite{GST}.
\end{proof}
\noindent Note that on the interior of the manifold we may work in the
scale $\mathring{g}=\si^{-2}\bm{g}$ whence the first integral
simplifies to $ \mathbf{u}^a \nabla^p k_{pa}$.

\bibliographystyle{spmpsci}
\bibliography{GST}
\end{document}